\documentclass[11pt,letterpaper]{amsart}
\usepackage{amsmath,amssymb,amsthm,eucal}
\usepackage{enumerate,multirow}
\usepackage{tikz}
\usepackage{placeins}
\usetikzlibrary{arrows}
\usetikzlibrary{decorations.pathreplacing}
\renewcommand{\phi}{\varphi}
\DeclareMathOperator{\Arg}{Arg}
\DeclareMathOperator{\Impart}{Im}
\DeclareMathOperator{\Repart}{Re}
\renewcommand{\Im}{\Impart}
\renewcommand{\Re}{\Repart}
\newcommand{\C}{{\mathbb C}}
\newcommand{\N}{{\mathbb N}}
\newcommand{\R}{{\mathbb R}}
\newcommand{\Z}{{\mathbb Z}}
\newcommand{\cA}{\mathcal A}
\newcommand{\cAk}{\cA_k}
\newcommand{\cAl}{\cA_\ell}
\newcommand{\cAthree}{\cA_3}
\newcommand{\cAfour}{\cA_4}
\newcommand{\cE}{\mathcal E}
\newcommand{\cP}{\mathcal P}
\newcommand{\cV}{\mathcal V}
\newcommand{\cW}{\mathcal W}
\newcommand{\cX}{\mathcal X}
\newcommand{\cY}{\mathcal Y}
\newcommand{\SPV}{S_\cP(V)}
\newcommand{\SPVn}{S_\cP(V_n)}
\newcommand{\SPn}{S_\cP(n)}
\newcommand{\SAkV}{S_{\cAk}(V)}
\newcommand{\SAkn}{S_{\cAk}(n)}
\newcommand{\SAln}{S_{\cAl}(n)}
\newcommand{\SAthreen}{S_{\cAthree}(n)}
\newcommand{\SAfourn}{S_{\cAfour}(n)}
\newcommand{\SEV}{S_\cE(V)}
\newcommand{\SEn}{S_\cE(n)}
\newcommand{\card}[1]{\left|{#1}\right|}
\newcommand{\conj}[1]{\overline{#1}}
\newcommand{\maxs}[1]{\max_{\substack{#1}}}
\newcommand{\floor}[1]{\lfloor{#1}\rfloor}
\newcommand{\ceil}[1]{\lceil{#1}\rceil}
\newtheorem{theorem}{Theorem}[section]
\newtheorem{proposition}[theorem]{Proposition}
\newtheorem{lemma}[theorem]{Lemma}
\newtheorem{construction}[theorem]{Construction}
\theoremstyle{remark}
\newtheorem{remark}[theorem]{Remark}
\newtheorem{example}[theorem]{Example}

\title{On the Number of Similar Instances of a Pattern in a Finite Set}

\author[\'Abrego]{Bernardo \'Abrego}
\address{Department of Mathematics, California State University, Northridge, \: United States}

\author[Fern\'andez-Merchant]{Silvia Fern\'andez-Merchant}

\author[Katz]{Daniel J.~Katz}

\author[Kolesnikov]{Levon Kolesnikov}

\date{20 November 2016}

\begin{document}
\begin{abstract}
New bounds on the number of similar or directly similar copies of a pattern within a finite subset of the line or the plane are proved.
The number of equilateral triangles whose vertices all lie within an $n$-point subset of the plane is shown to be no more than $\lfloor{(4 n-1)(n-1)/18}\rfloor$.
The number of $k$-term arithmetic progressions that lie within an $n$-point subset of the line is shown to be at most $(n-r)(n+r-k+1)/(2 k-2)$, where $r$ is the remainder when $n$ is divided by $k-1$.
This upper bound is achieved when the $n$ points themselves form an arithmetic progression, but for some values of $k$ and $n$, it can also be achieved for other configurations of the $n$ points, and a full classification of such optimal configurations is given.
These results are achieved using a new general method based on ordering relations.
\end{abstract}
\maketitle
\section{Introduction}
Erd\H os and Purdy \cite{Erdos-Purdy-1971,Erdos-Purdy-1975,Erdos-Purdy-1976} raised the question of finding the maximum number of equilateral triangles that can be determined by $n$ points in the plane, where we say that a triangle is determined by a set $V$ when the three vertices of the triangle lie in $V$.
This problem is also mentioned in the compendia of unsolved problems in geometry by Croft, Falconer, and Guy \cite{Croft-Falconer-Guy} and by Brass, Moser, and Pach \cite{Brass-Moser-Pach}, and is discussed by Pach and Sharir in \cite{Pach-Sharir-2009}.
There are many other variations of this problem involving other patterns, constraints on the $n$-set in which instances of the pattern are sought, higher-dimensional ambient spaces, different definitions for what counts as an instance of a pattern, and different optimization objectives, e.g., \cite{Abrego-Elekes-Fernandez,Abrego-Fernandez-2002-Convex,Abrego-Fernandez-Llano,Agarwal-Apfelbaum-Purdy-Sharir,Apfelbaum-Sharir-2005,Brass-2000,Burton-Purdy,Cortier-Goaoc-Lee-Na,Pach-Pinchasi,Pach-Sharir-1992,vanKreveld-deBerg,Xu-Ding}.
These problems trace their inspiration to Erd\H os' question \cite{Erdos-1946} about the maximum number of pairs of points at unit distance that can be determined by an $n$-subset of the plane, and other related questions, which have led to a rich literature (see \cite{Brass-Moser-Pach,Pach-Agarwal} for an overview).
Apart from being a well-known and important question in discrete geometry, the problem of determining the number of instances of a pattern in an $n$-subset of the plane is relevant to the problem of pattern recognition in data from scanners, cameras, and telescopes \cite{Brass-2002,Brass-Moser-Pach,Brass-Pach}.

In all these problems, we have a {\it universe}, which is a set $U$, an equivalence relation $\sim$ on the power set of $U$, and a {\it pattern} $\cP$, which is an equivalence class of $\sim$.
If $P \in \cP$, we say that $P$ is an {\it instance} of pattern $\cP$.
We call $\cP$ a $k$-pattern if all its instances are $k$-sets.
For the rest of this paper, instances of the pattern $\cP$ are always assumed to be finite sets.
For $V \subseteq U$, we let 
\[
\SPV=\card{\{P \subseteq V: P \in \cP\}}.
\]
We are concerned with finite subsets $V$ of our universe $U$.
For an integer $n$, we define 
\[
\SPn=\maxs{V \subseteq U \\ \card{V}=n} \SPV,
\]
that is, the largest number of instances of $\cP$ that can be found in an $n$-subset of $U$.

The results of this paper concern the case when $U$ is the line $\R$ or the plane $\R^2$ (identified with $\C$), where $\sim$ is the geometric relation of similarity or direct similarity,\footnote{Two figures are said to be {\it directly similar} if one can be obtained from the other by a rotation and a translation.  Reflections are not allowed.} and where the patterns $\cP$ are classes of finite subsets, such as arithmetic progressions in the line, or equilateral triangles in the plane.

If our relation $\sim$ is similarity or direct similarity, and the instances of our pattern $\cP$ have more than one point, and we are counting the number of instances of $\cP$ in subsets of the plane, then it is not hard to show that $\SPn$ is $O(n^2)$.
Elekes and Erd\H os \cite{Elekes-Erdos} proved a subquadratic lower bound for general patterns $\cP$ and a quadratic lower bound if there is an instance of $\cP$ whose points all have algebraic numbers for their coordinates.
Laczkovich and Ruzsa \cite{Laczkovich-Ruzsa} later showed that $\SPn=\Theta(n^2)$ if and only if the cross ratio of every quadruplet of distinct points in an instance of $\cP$ is algebraic.
However, before this work there were no patterns $\cP$ (with at least three points) for which the quadratic coefficient was known.

Let $\cE$ be the pattern of the vertices of an equilateral triangle, that is, $\cE$ contains all $3$-point subsets of the plane such that all three pairs of points within the set have the same distance.
Then the results of Laczkovich and Ruzsa show that $\SEn=\Theta(n^2)$.
Prior to this paper, the best known bounds \cite{Abrego-Fernandez-2000} for $\SEn$ were 
\begin{equation}\label{William}
\left(\frac{1}{3}-\frac{\sqrt{3}}{4 \pi}\right) n^2 + O(n^{3/2}) \leq \SEn \leq \left\lfloor\frac{(n-1)^2}{4}\right\rfloor,
\end{equation}
so that 
\[
0.1955 < \frac{1}{3}-\frac{\sqrt{3}}{4 \pi} \leq \liminf_{n\to\infty} \frac{\SEn}{n^2} \leq \limsup_{n\to\infty} \frac{\SEn}{n^2} \leq \frac{1}{4}.
\]
The lower bound is obtained from the points in the equilateral triangle lattice
contained in a disk of a suitable radius.
It is conjectured \cite[Conjecture 1]{Abrego-Fernandez-2000} that this construction is asymptotically optimal and that $\lim_{n\to\infty} \SEn/n^2$ exists and equals $\frac{1}{3}-\frac{\sqrt{3}}{4 \pi}$, the lower endpoint of the interval in which the limits inferior and superior for $\SEn/n^2$ are known to lie.

The main result of this paper is an improved upper bound on $\SEn$.
\begin{theorem}\label{Katherine}
If $\cE$ is the pattern of the vertices of equilateral triangles in $\R^2$, then $\SEn \leq \floor{(4 n-1)(n-1)/18}$.
\end{theorem}
Therefore $\limsup_{n\to\infty} \SEn/n^2 \leq 2/9$.
The proof is given in Section \ref{Eric}, and is based on an order-theoretic methodology, developed in Sections \ref{Orestes}--\ref{Gerald}, for attacking pattern-counting problems.

Our new methodology was first developed while exploring the simpler problem of counting instances of finite patterns (such as arithmetic progressions) in a finite subset of the line $\R$, and was used to discover results that are presented in Sections \ref{Larry}--\ref{Rebecca}.
After settling preliminaries in Section \ref{Larry}, we calculate the precise maximum number of $k$-term arithmetic progressions that can occur in an $n$-point subset of the line in Section \ref{Christopher}.
The following is the main result of Section \ref{Christopher}.
\begin{theorem}\label{Thomas}
Let $k > 1$, and let $\cAk$ be the pattern of $k$-term arithmetic progressions on $\R$.
If $n \in \N$ and $r$ is the remainder when $n$ is divided by $k-1$, then $\SAkn= (n-r)(n+r-k+1)/(2 k-2)$.
\end{theorem}
For each $k > 2$, we completely classify all $n$-point subsets of the line that contain as many $k$-term arithmetic progressions as possible.
These optimal sets include $n$-term arithmetic progressions, but there are some other optimal sets: when $k > 3$ and $k-1$ divides $n$, sets obtained by deleting the second or penultimate point from an arithmetic progression are also optimal (see Proposition \ref{Oliver}).
When $k=3$, the variety of the optimal sets is considerably richer (see Proposition \ref{Imogene}).
This work complements the asymptotic approach of Elekes \cite{Elekes}, who studied the structure of $n$-subsets of the line that contain an asymptotically large number of $k$-term arithmetic progressions, whereas we determine the precise maximum and those sets that achieve it.

In Section \ref{Rebecca}, we show that our method can be used to obtain both upper and lower bounds on $\SPn$ when $\cP$ is a commensurable pattern in the line, that is, where the ratios of the distances between points in $\cP$ are all rational.
We count the number of directly similar copies of $P=\{0,1,3\}$ in finite subsets of the line, and thereby show that if $\cP$ is the pattern for $P$, then
\begin{equation}\label{Howard}
\frac{1}{6} \leq \liminf_{n\to\infty} \frac{\SPn}{n^2} \leq \limsup_{n\to\infty} \frac{\SPn}{n^2} \leq \frac{1}{4}.
\end{equation}
We then construct a family of $n$-point subsets $V_n$ of the line that has $\lim_{n\to\infty} \SPVn/n^2 = 3/16$, which is the highest known value and, interestingly, lies strictly between the bounds of \eqref{Howard}.

As noted, our order-theoretic method is a new tool for counting patterns in finite sets.
It combines nicely with other techniques of discrete geometry, for example, the proof of Theorem \ref{Katherine} combines our tools with the topological technique of finding three concurrent halving lines of a set (cf.~\cite[Lemma 2]{Fekete-Meijer}, \cite[Lemma 3]{Erickson-Hurtado-Morin}, and \cite[Lemma 2]{Dumitrescu-Pach-Toth}).
Further applications will be explored in future works.

\section{Order and Decompositions}\label{Orestes}

A key theme in this work is that one can place an order relation on the universe $U$, and use it to create bounds on $\SPn$.
So henceforth, we assume that our universe $U$ has a total ordering relation $\preceq$.
If $V \subseteq U$, then we say that a point $v \in V$ is the {\it $i$th point of $V$ with respect to $\preceq$} to mean that there are precisely $i-1$ points $w \in V$ with $w \prec v$.
If $V, W \subseteq U$, then we write $V \prec W$ to mean that $v \prec w$ for every $v \in V$ and $w \in W$.

If $V$ is a set and $\ell$ a positive integer, then an {\it $\ell$-decomposition} of $V$ is a set $\{V_1,\ldots,V_\ell\}$ of disjoint subsets of $V$ such that $\bigcup_{j=1}^\ell V_j=V$.  (So a decomposition differs from a partition only inasmuch as the former may include empty subsets.)
An $\ell$-decomposition of a finite set $V$ with $n$ points is said to be {\it balanced} if each set in the decomposition has either $\floor{n/\ell}$ or $\ceil{n/\ell}$ points in it.
An $\ell$-decomposition $\cV$ of $V$ is called {\it $\preceq$-orderly} if, whenever $V, V^\prime$ are distinct elements of $\cV$, then either $V\prec V^\prime$ or $V^\prime \prec V$.
If $\cV$ is an orderly $\ell$-decomposition of $V$, then we say that $V \in \cV$ is the $i$th set in $\cV$ to mean that there are precisely $i-1$ sets $W \in \cV$ with $W \prec V$.

If $k \geq 2$ and there is a $\preceq$-orderly $(k-1)$-decomposition $\cV$ of $V$, and a $k$-subset $P$ of $V$, then we say that $P$ is of {\it echelon $j$ in $\cV$} if the $j$th and $(j+1)$th points of $P$ lie in the $j$th set of $\cV$.
That is, if $\cV=\{V_1 \prec \cdots \prec V_{k-1}\}$ and $P=\{p_1 \prec \cdots \prec p_k\}$, then $P$ is of echelon $j$ in $\cV$ if $p_j, p_{j+1} \in V_j$.
It is a key fact that $P$ must be of echelon $j$ in $\cV$ for at least one $j$.
\begin{lemma}\label{Alice}
Let $k\geq 2$ and let $V$ be a set with a $\preceq$-orderly $(k-1)$-decomposition $\cV$.
If $P$ is a $k$-subset of $V$, then $P$ is of echelon $j$ in $\cV$ for some $j \in \{1,\ldots,k-1\}$.
\end{lemma}
\begin{proof}
By induction on $k$, with the $k=2$ case trivial.
For $k > 2$, suppose that $P=\{p_1 \prec \cdots \prec p_k\}$ is not of echelon $(k-1)$ in $\cV=\{V_1 \prec \cdots \prec V_{k-1}\}$.
Then $p_{k-1} \not \in V_{k-1}$, and thus $P\smallsetminus\{p_k\}$ is a $(k-1)$-subset of $V\smallsetminus V_{k-1}$ with $\preceq$-orderly $(k-2)$-decomposition $\{V_1,\ldots,V_{k-2}\}$, and hence there is some $j \in \{1,\ldots,k-2\}$ such that $p_j,p_{j+1} \in V_j$ by induction.
\end{proof}
\begin{remark}
It should be noted that Lemma \ref{Alice} is somewhat more detailed in its conclusions than the usual Pigeonhole Principle.  For example, suppose that our universe $U$ is $\R$ and $\preceq$ is the usual ordering $\leq$ of real numbers.  Let $V=P=\{p_1 < \ldots < p_6\}$, and let $\cV=\{V_1,\ldots,V_5\}$ be the orderly $5$-decomposition of $V$ shown in Figure \ref{Aaron}.
\FloatBarrier
\begin{center}
\begin{figure}
\begin{center}
\begin{tikzpicture}[scale=0.85]
\fill (1,0.5) circle [radius=2pt];
\fill (3.5,0.5) circle [radius=2pt];
\fill (4.5,0.5) circle [radius=2pt];
\fill (7,0.5) circle [radius=2pt];
\fill (8,0.5) circle [radius=2pt];
\fill (10.5,0.5) circle [radius=2pt];
\node (p1) at (1,0) {$p_1$};
\node (p2) at (3.5,0) {$p_2$};
\node (p3) at (4.5,0) {$p_3$};
\node (p4) at (7,0) {$p_4$};
\node (p5) at (8,0) {$p_5$};
\node (p6) at (10.5,0) {$p_6$};
\draw (-1.5,0.25) ellipse (1.0 and 0.6);
\draw (1,0.25) ellipse (1.0 and 0.6);
\draw (4,0.25) ellipse (1.5 and 0.6);
\draw (7.5,0.25) ellipse (1.5 and 0.6);
\draw (10.5,0.25) ellipse (1.0 and 0.6);
\node (V1) at (-1.5,-0.75) {$V_1$};
\node (V2) at (1.0,-0.75) {$V_2$};
\node (V3) at (4.0,-0.75) {$V_3$};
\node (V4) at (7.5,-0.75) {$V_4$};
\node (V5) at (10.5,-0.75) {$V_5$};
\end{tikzpicture}
\caption{Illustration of Lemma \ref{Alice}}\label{Aaron}
\end{center}
\end{figure}
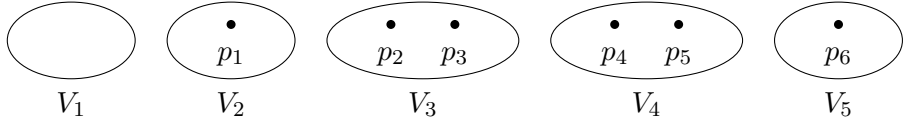
\end{center}
Since we have six points in five sets, the usual Pigeonhole Principle shows that there is at least one set in $\cV$ containing at least two points of $P$; this is seen in $V_3$, which contains $p_2$ and $p_3$.  But $P=\{p_1 < \ldots < p_6\}$ (which is the same as $V$) is not of echelon $i=3$ in $V$, because although $V_3$ contains two points of $P$, it does not contain the particular points $p_i=p_3$ and $p_{i+1}=p_4$.  On the other hand, Lemma \ref{Alice} asserts that there is some $j \in \{1,\ldots,5\}$ such that $P$ is of echelon $j$ in $V$, and we see $P$ is indeed of echelon $j=4$ since points $p_j=p_4$ and $p_{j+1}=p_5$ lie in $V_4$.

Of course, if $k \geq 4$, and $V$ is a set with $\preceq$-orderly $(k-1)$-decomposition $\cV$, then a $k$-subset $P$ of $V$ may be of echelon $j$ in $\cV$ for more than one $j$.
\end{remark}
\section{Reconstructibility}

Let $\cP$ be a $k$-pattern in our universe $U$ with order relation $\preceq$.
Suppose we are given some $j \in \{1,\ldots,k-1\}$ and two points $u, v \in U$ with $u\prec v$.
Any $P \in \cP$ that has $u$ and $v$ respectively as its $j$th and $(j+1)$th points is called a {\it reconstruction of $\cP$} with the prescribed points as its $j$th and $(j+1)$th points.

If, for every $j \in \{1,\ldots,k-1\}$ and $u \prec v \in U$, there is at least one reconstruction of $\cP$ with $u$ and $v$ as $j$th and $(j+1)$th points, then we say that $\cP$ {\it admits at least one reconstruction from $\preceq$-consecutive points}.
If there is at most one reconstruction for every  $j \in \{1,\ldots,k-1\}$ and $u \prec v \in U$, we say that $\cP$ {\it admits at most one reconstruction from $\preceq$-consecutive points}.
If both of these hold, then we say that $\cP$ {\it admits a unique reconstruction from $\preceq$-consecutive points}, or {\it is uniquely reconstructible from $\preceq$-consecutive points}.

\begin{example}
Suppose that our pattern $\cP$ is the set of all triples that are vertices of isosceles right triangles, and suppose that our order relation $\preceq$ is lexicographic ordering of $\R^2$.  (That is, $(a,b) \prec (c,d)$ if and only if either (i) $a < b$ or (ii) $a=b$ and $c < d$.)  We see that a pair of points $A$ and $B$ may be the first and second points (under the relation $\preceq$) of more than one instance of $\cP$.  For instance, if $A$ and $B$ are the points in Figure \ref{Bartholomew} below, there are three isosceles right triangles of which $A$ and $B$ are the leftmost two vertices, namely, $\bigtriangleup A B C_1$, $\bigtriangleup A B C_2$, and $\bigtriangleup A B C_3$.  (There are also three other isosceles right triangles, $\bigtriangleup A B C_4$, $\bigtriangleup A B C_5$, and $\bigtriangleup A B C_6$, with $A$ and $B$ as vertices, but $A$ and $B$ are not the leftmost two vertices in these.)  If one looks at other pairs $A$ and $B$, one can find out that $\cP$ always admits at least one reconstruction from $\preceq$-consecutive points, but often admits more than one reconstruction.
\begin{center}
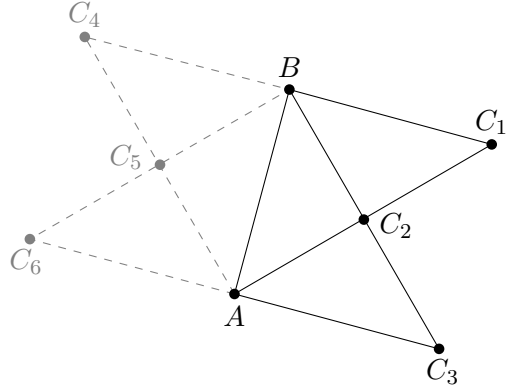
\begin{figure}
\begin{center}
\begin{tikzpicture}
\node (A) at  ( 0   , 0   ) {};
\node (B) at  ( 0.73, 2.72) {};
\node (C1) at ( 3.42, 1.99) {};
\node (C2) at ( 1.72, 0.99) {};
\node (C3) at ( 2.72,-0.73) {};
\node (C4) at (-1.99, 3.42) {};
\node (C5) at (-0.99, 1.72) {};
\node (C6) at (-2.72, 0.73) {};
\draw (0,0)--(0.73,2.72);
\draw (0,0)--(3.42,1.99)--(0.73,2.72);
\draw (0,0)--(2.72,-0.73)--(0.73,2.72);
\draw [gray,dashed] (0,0)--(-1.99,3.42)--(0.73,2.72);
\draw [gray,dashed] (0,0)--(-2.72,0.73)--(0.73,2.72);
\fill (A) circle [radius=2pt];
\fill (B) circle [radius=2pt];
\fill (C1) circle [radius=2pt];
\fill (C2) circle [radius=2pt];
\fill (C3) circle [radius=2pt];
\fill (C4) [gray] circle [radius=2pt];
\fill (C5) [gray] circle [radius=2pt];
\fill (C6) [gray] circle [radius=2pt];
\node (AL) at (0,-0.3) {$A$};
\node (BL) at (0.73,3.02) {$B$};
\node (C1L) at (3.42,2.29) {$C_1$};
\node (C2L) at (2.15,0.91) {$C_2$};
\node (C3L) at (2.77,-1.03) {$C_3$};
\node (C4L) at (-1.99, 3.72) [gray] {$C_4$};
\node (C5L) at (-1.44, 1.82) [gray] {$C_5$};
\node (C6L) at (-2.77, 0.43) [gray] {$C_6$};
\end{tikzpicture}
\caption{Points $A$ and $B$ are vertices in six isosceles right triangles, three of which are reconstructions that make $A$ and $B$ the leftmost two vertices.}\label{Bartholomew}
\end{center}
\end{figure}
\end{center}
\end{example}
\FloatBarrier
\begin{example}
Suppose that our pattern $\cP$ is the set of all triples that are vertices of equilateral triangles, and suppose our order relation $\preceq$ is lexicographic ordering of $\R^2$.  We see that a pair of points $A$ and $B$ might not be the second and third points (under the relation $\preceq$) of any instance of $\cP$.  For instance, if $A$ and $B$ are the points in Figure \ref{Caesar} below, there is no equilateral triangle of which $A$ and $B$ are the rightmost two vertices.  (There are two equilateral triangles, $\bigtriangleup A B C_1$ and $\bigtriangleup A B C_2$, with $A$ and $B$ as vertices, but $A$ and $B$ are not the rightmost two vertices in these.)  If one looks at other pairs $A$ and $B$, one can find out that $\cP$ never admits more than one reconstruction from $\preceq$-consecutive points: it sometimes admits no reconstruction, and sometimes admits one reconstruction.  In Lemma \ref{Jacqueline}, we determine precisely when these two cases occur.
\begin{center}
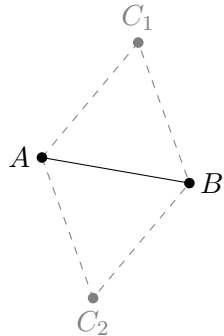
\begin{figure}
\begin{center}
\begin{tikzpicture}
\node (A) at  (-0.98, 0.17) {};
\node (B) at  ( 0.98,-0.17) {};
\node (C1) at ( 0.30, 1.70) {};
\node (C2) at (-0.30,-1.70) {};
\draw (-0.98,0.17)--(0.98,-0.17);
\draw [gray,dashed] (-0.98,0.17)--(0.30,1.70) --(0.98,-0.17);
\draw [gray,dashed] (-0.98,0.17)--(-0.30,-1.70) --(0.98,-0.17);
\fill (A) circle [radius=2pt];
\fill (B) circle [radius=2pt];
\fill (C1) [gray] circle [radius=2pt];
\fill (C2) [gray] circle [radius=2pt];
\node (LA) at  (-1.28, 0.17) {$A$};
\node (LB) at  ( 1.28,-0.17) {$B$};
\node (LC1) at ( 0.30, 2.03) [gray] {$C_1$};
\node (LC2) at (-0.30,-2.03) [gray] {$C_2$};
\end{tikzpicture}
\caption{Points $A$ and $B$ are vertices in two equilateral triangles, but neither of these is a reconstruction that makes $A$ and $B$ the rightmost two vertices.}\label{Caesar}
\end{center}
\end{figure}
\end{center}
\end{example}
\begin{example}
Suppose that our pattern $\cP$ consists of all sets of points in $\R$ that are directly similar to $\{0,1,3,6\}$, and suppose that our order relation $\preceq$ is the usual ordering of $\R$.  A pair of points $A$ and $B$ with $A < B$ will always be the second and third points of precisely one instance of $\cP$, namely, $\{(3 A-B)/2,A,B,(5 B-3 A)/2\}$.  For example, if $A$ and $B$ are $-1$ and $5$, respectively, then the unique instance of $\cP$ of which $A$ and $B$ are the second and third points is $\{-4,-1,5,14\}$.
\end{example}
\section{A General Upper Bound on $\SPn$}\label{Gerald}

\begin{theorem}\label{Beatrice}
Let $k > 1$ and let $\cP$ be a $k$-pattern in $U$ that admits at most one reconstruction from $\preceq$-consecutive points.
If $n \in \N$ and $r$ is the remainder when $n$ is divided by $k-1$, then $\SPn \leq (n-r)(n+r-k+1)/(2 k-2)$.
\end{theorem}
\begin{proof}
Let $V$ be an $n$-subset of $U$, and let $\cV=\{V_1 \prec \cdots \prec V_{k-1}\}$ be a $\preceq$-orderly, balanced $(k-1)$-decomposition of $V$.
By Lemma \ref{Alice}, each $P \subseteq V$ with $P \in \cP$ is of echelon $j$ in $\cV$ for some $j \in \{1,\ldots,k-1\}$.
Each subset of echelon $j$ has its $j$th and $(j+1)$st elements in $V_j$, and since no two instances of $\cP$ may have the same $j$th and $(j+1)$th points, this means that there are at most $\binom{\card{V_j}}{2}$ instances of $\cP$ that are of echelon $j$.
Thus there are at most $\sum_{j=0}^{k-1} \binom{\card{V_j}}{2}$ instances of $\cP$ in $V$.
If we write $n=q (k-1)+r$, then our decomposition, being balanced, has $\card{V_j}=q+1$ for $r$ values of $j$ and $\card{V_j}=q$ for $k-1-r$ values of $j$, so that $\SPV \leq r\binom{q+1}{2}+(k-1-r)\binom{q}{2}=(n-r)(n+r-k+1)/(2 k-2)$.
\end{proof}
The method of counting used in this proof provides a criterion for when this upper bound is achieved.
\begin{lemma}\label{Francis}
Let $k > 1$ and let $\cP$ be a $k$-pattern in $U$ that admits at most one reconstruction from $\preceq$-consecutive points.
Let $n \in \N$ and let $r$ be the remainder when $n$ is divided by $k-1$.
Let $V$ be an $n$-subset of $U$ with $\preceq$-orderly, balanced $(k-1)$-decomposition $\cV=\{V_1,\ldots,V_{k-1}\}$.
Then $\SPV=(n-r)(n+r-k+1)/(2 k-2)$ if and only if for each $j \in \{1,\ldots,k-1\}$ and every pair of distinct points $v, w \in V_j$, there exists a reconstruction $P$ of $\cP$ with $v$ and $w$ as $j$th and $(j+1)$th points, and $P$ is a subset of $V$ that is not of echelon $i$ in $\cV$ for any $i\not=j$.
\end{lemma}
\begin{proof}
Examining the proof of the preceding theorem, we see that if any of the reconstructions mentioned in the statement of this proposition did not exist, or did not lie in $V$, then $\binom{\card{V_j}}{2}$ would be an overestimate of the number of instances of $\cP$ that are of echelon $j$.
And if any instance were of two distinct echelons in $\cV$, then we would be counting it twice in $\sum_{j=0}^{k-1} \binom{\card{V_j}}{2}$, thus making our bound an overestimate.
\end{proof}

\section{General Patterns in the Line}\label{Larry}
In this section, our universe $U$ is the line $\R$, our equivalence relation $\sim$ is direct similarity, and our pattern $\cP$ can be the equivalence class for any finite subset of $\R$.
Our order relation is the usual order relation $\leq$ for $\R$.
We note that any pattern $\cP$ is uniquely reconstructible from $\leq$-consecutive points, so Theorem \ref{Beatrice} applies, amounting to the following.
\begin{theorem}\label{Deborah}
Let $k > 1$ and let $\cP$ be any $k$-pattern in $\R$.  
If $n \in \N$ and $r$ is the remainder when $n$ is divided by $k-1$, then $\SPn \leq (n-r)(n+r-k+1)/(2 k-2)$.
\end{theorem}
This upper bound is tight in the sense that for each $k>1$, when $\cP$ is the pattern of $k$-term arithmetic progressions, then $\SPn$ equals the bound, as we shall now show.

\section{Arithmetic Progressions in the Line}\label{Christopher}

Here the universe $U$ is the line $\R$ and our pattern is $\cAk$, the class of $k$-term arithmetic progressions.
Because $\cAk$ is invariant under reflection, the results are the same whether we make $\sim$ direct similarity or similarity.
This enables an exact computation of $\SAkn$: we recall and prove Theorem \ref{Thomas}.
\begin{theorem}\label{George}
Let $k > 1$, and let $\cAk$ be the pattern of $k$-term arithmetic progressions on $\R$.
If $n \in \N$ and $r$ is the remainder when $n$ is divided by $k-1$, then $\SAkn= (n-r)(n+r-k+1)/(2 k-2)$.
\end{theorem}
\begin{proof}
This is an immediate consequence of Theorem \ref{Deborah} above and Lemma \ref{Eustace} below.
\end{proof}
\begin{lemma}\label{Eustace}
If $k > 1$ and $n \in \N$ with $r$ the remainder when $n$ is divided by $k-1$, then the number of $k$-term arithmetic progressions in $\{0,1,\ldots,n-1\}$ is $(n-r)(n+r-k+1)/(2 k-2)$.
\end{lemma}
\begin{proof}
For each positive integer $s \leq n/(k-1)$, our set $\{0,1,\ldots,n-1\}$ contains precisely $n-s(k-1)$ arithmetic progressions having $k$ points with distance $s$ between consecutive points.  The result follows by adding these quantities for $1 \leq s \leq n/(k - 1)$.
\end{proof}
Now that we know the precise value of $\SAkn$, we would like to completely classify the $n$-subsets $V$ of $\R$ achieving $\SAkV=\SAkn$.
We call such $n$-sets {\it optimal} for $k$-term arithmetic progressions.
Lemma \ref{Eustace} above shows that $n$-term arithmetic progressions are always optimal for $k$-term arithmetic progressions, but in many cases there are other optimal sets, and we now classify them (up to similarity).

All sets are trivially optimal for $1$- or $2$-term arithmetic progressions, and any $n$-set with $n < k$ is trivially optimal for $k$-term arithmetic progressions.
The {\it barycenter} of an arithmetic progression is the arithmetic mean of its points.
Two progressions in $\Z$ are said to be {\it concentric} if they have the same barycenter, or {\it nearly concentric} if their barycenters differ by $1$.
\begin{proposition}\label{Imogene}
For $n \geq 3$, an $n$-subset of $\R$ is optimal for $3$-term arithmetic progressions if and only if, up to similarity, it is the union $E \cup O$ of two nonempty arithmetic progressions, where $E$ consists of consecutive even integers and $O$ consists of consecutive odd integers and $E$ and $O$ are concentric if $n$ is odd, or nearly concentric if $n$ is even.
\end{proposition}
\begin{remark}
Before we embark upon the proof, we pause to note that the set $E\cup O$ described here is an arithmetic progression when $\card{E}$ and $\card{O}$ are either equal, or one differs from the other by $1$.
\end{remark}
{\it Proof of Proposition \ref{Imogene}}: In view of the value of $\SAthreen$ from Theorem \ref{George}, we may use Lemma \ref{Francis} as the criterion for optimality.
Suppose that $V$ is an $n$-subset of $\R$ and $\cV=\{V_1 < V_2\}$ is a $\leq$-orderly, balanced $2$-decomposition of $V$.
Because arithmetic progressions are uniquely reconstructible from consecutive points, and since no $3$-term progression may be both of echelon $1$ and of echelon $2$ (since that would require it to have more than three points), Lemma \ref{Francis} shows that $V$ is optimal for $3$-term arithmetic progressions if and only if for every $j \in \{1,2\}$ and every pair of distinct $v, w \in V_j$, the reconstruction of $\cAthree$ with $v$ and $w$ as $j$th and $(j+1)$th points is contained in $V$.

First, suppose $V$ has the form $E \cup O$ as described in the statement of this proposition.
If $\card{E}$ and $\card{O}$ are equal or differ by $1$, then $V$ is itself an arithmetic progression, and hence optimal by Lemma \ref{Eustace}, so we may assume that $\card{E}\not=\card{O}$.
We may then view $V$ as $X \cup Y$, where $X$ is the larger of $E$ and $O$, while $Y$ contains the smaller of $E$ and $O$, along with any points from the larger of $E$ and $O$ that lie between the points of the smaller of $E$ and $O$.
Thus $X$ is an arithmetic progression with distance $2$ between its consecutive points, while $Y$ is an arithmetic progression with distance $1$ between its consecutive points.
We let $\cV=\{V_1 < V_2\}$ be an $\leq$-orderly, balanced $2$-decomposition of $V$.
Then $\cV$ induces $\leq$-orderly, balanced $2$-decompositions of $X$ and $Y$, which are $\cX=\{X\cap V_1, X\cap V_2\}$ and $\cY=\{Y\cap V_1, Y\cap V_2\}$, respectively.
Now suppose that $v_1 < v_2$ are points in $V_1$, and let $v_3$ be the point that makes $\{v_1 < v_2 < v_3\}$ a $3$-term arithmetic progression.
If $v_1$ and $v_2$ are both in $X_1$ (resp., $Y_1$), then since $X$ (resp., $Y$) is an arithmetic progression, and hence optimal for $3$-term arithmetic progressions, we may apply Lemma \ref{Francis} to conclude that $v_3 \in X$ (resp., $Y$), and hence $v_3 \in V$.
In the remaining case, where $v_1 \in X_1$ and $v_2 \in Y_1$, we note that $v_3$ cannot be larger than the largest point in $V$ and has the same parity as $v_1$, hence $v_3 \in X \subseteq V$.
By the same argument, if $v_2 < v_3$ are two points in $V_2$, then the point $v_1$ that makes $\{v_1 < v_2 < v_3\}$ a $3$-point arithmetic progression is always in $V$.
Thus our set $V$ satisfies the sufficient criteria of Lemma \ref{Francis} for optimality.

Now suppose that $V$ is an $n$-subset of $\R$ that is optimal for $3$-term arithmetic progressions.
We want to show that $V$ has the form described in the statement of this proposition.
First we deal with the case where $n$ is odd, and without loss of generality, we may apply a similarity transformation so that the middle point of $V$ is $0$, and the next point in $V$ is $1$.
We let $\cV=\{V_1 < V_2\}$ be an $\leq$-orderly, balanced $2$-decomposition of $V$ with the negative points in $V_1$, and the nonnegative points in $V_2$.
By Lemma \ref{Francis}, for every positive $v_2 \in V_2$, the $3$-term arithmetic progression $-v_2 < 0 < v_2$ must lie in $V$, so the point $-v_2$ must lie in $V_1$.
Thus $V$ has reflection symmetry about $0$.

If $u < v < w$ are consecutive points in $V_2$, then by Lemma \ref{Francis}, the $3$-term progression $2 v- w < v < w$ must lie in $V$, and so $2 v - w \leq u$, that is, $w-v \geq v-u$.
So the spacings between consecutive points of $V_2$ are nondecreasing as we proceed to the right.

If $1 \leq u < v$ are consecutive points in $V_2$, then by Lemma \ref{Francis}, the $3$-term progression $2-v < 1 < v$ lies in $V$, so that $v-2$ lies in $V$ by reflection symmetry, and so $u \geq v-2$.
So the spacing between any two consecutive points of $V_2$ is no greater than $2$.

We also claim that all the points of $V_2$ are integers: we proceed by induction.
The first two elements of $V_2$ are $0$ and $1$, so suppose that $v \in V_2$ with $v> 1$, and we know that all elements of $V_2$ less than $v$ are integers.
By Lemma \ref{Francis}, the arithmetic progression $2-v < 1 < v$ lies in $V$, so that $|v-2|$ lies in $V$ by reflection symmetry, so $|v-2|$ must be an element of $V_2$ less than $v$, hence an integer.  So $v$ is an integer.

Thus spacings between any consecutive points of $V_2$ is always $1$ or $2$, and these spacings are nondecreasing as we proceed to the right.
Since $V$ has reflection symmetry about $0$, we know that $V$ has the form described in this proposition.
This completes the proof when $n$ is odd.

Now suppose that $n$ is even and $V$ is an $n$-subset of $\R$ optimal for $3$-term arithmetic progressions.
Then the technical Lemma \ref{Hannah}, whose proof is delayed to the end of this section, shows that an $(n-1)$-set $V^\prime$ obtained from $V$ by removing either the leftmost or rightmost point is optimal for $3$-term arithmetic progressions.
By applying similarity transformations, we may assume that the rightmost point $r$ is removed and that $V^\prime$ has the form described in the statement of this proposition, with $0$ being the middle point of $V^\prime$.
Then $V^\prime$ contains the point $1$ since the set of odd points in $V^\prime$ is nonempty and centered at $0$.
Let $r^\prime$ be the rightmost point of $V^\prime$.
A $\leq$-orderly, balanced $2$-decomposition $\cV=\{V_1 < V_2\}$ of $V$ places all the nonpositive points into $V_1$ and all the positive points into $V_2$.

In view of Lemma \ref{Francis}, the $3$-term progression $2-r < 1 < r$ lies in $V$, so $2-r \in V^\prime$.
Thus $r$ must be an integer, and by reflection symmetry in $V^\prime$, we must have $r-2 \leq r^\prime$, that is $r\in \{r^\prime+1,r^\prime+2\}$.
 It is immediate that $V^\prime \cup \{r^\prime+2\}$ is of the form described in the proposition, so we are done if $r=r^\prime+2$.
If $r=r^\prime+1$, then by Lemma \ref{Francis}, the $3$-term progression $r^\prime-1 < r^\prime < r$ is contained in $V$, so then $V^\prime$ must consist of all consecutive integers from $-r^\prime$ to $r^\prime$, and then it is clear that $V=V^\prime \cup \{r\}$ is of the form described in the proposition. \hfill \qedsymbol
\begin{proposition}\label{Oliver}
For $n \geq k \geq 4$, an $n$-subset $V$ of $\R$ is optimal for $k$-term arithmetic progressions if and only if, up to similarity, $V$ is an arithmetic progression or $k-1$ divides $n$ and $V$ is an arithmetic progression with the second point deleted.
\end{proposition}
\begin{proof}
We first analyze all optimal sets $V$ for $4$-term arithmetic progressions. 
Suppose that $V$ is such an $n$-subset of $\R$ and $n \geq 4$. Let $\cV=\{V_1 < V_2 < V_3\}$ with $\card{V_2} \geq 2$ be an $\leq$-orderly, balanced $3$-decomposition of $V$.
The technical Lemma \ref{minusOneOptimal} below shows that if we remove $V_3$ from $V$, then we are left with an optimal set for $3$-term arithmetic progressions.
Since $V_1 \cup V_2$ is optimal for $3$-term arithmetic progressions, Proposition \ref{Imogene} tells us that without loss of generality, we may take $V_1\cup V_2$ to be the union of a set $E$ of consecutive even integers and a set $O$ of consecutive odd integers, which either have a common center $c$ (if $\card{E \cup O}$ is odd) or $E$ and $O$ have distinct centers $d$ and $d+1$ in some order.

We now prove that the distance between the leftmost two points of $V_2$ is $1$.
Since $V_1 \cup V_2 = E \cup O$ and $E$ and $O$ are nonempty, $V_1 \cup V_2$ contains a set of three consecutive integers.
The integers are $c-1, c, c+1$ if $\card{E \cup O}$ is odd or one of $d-1,d,d+1$ or $d,d+1,d+2$ if $\card{E \cup O}$ is even.
Then the three consecutive points are either the two rightmost of $V_1$ and the one leftmost of $V_2$ or the one rightmost of $V_1$ and the two leftmost of $V_2$.
The first case implies the second one, since if the rightmost two points of $V_1$ have a spacing $1$, then the leftmost two points of $V_2$ will complete a $4$-term arithmetic progression with the former two points.
So we know that distance between the leftmost two points of $V_2$ is $1$.

By Lemma \ref{minusOneOptimal}, the union of $V_2$ and $V_3$ is optimal for $3$-term arithmetic progressions.
Now since the leftmost two points of $V_2 \cup V_3$ are at distance $1$, it follows from Proposition \ref{Imogene} that $V_2 \cup V_3$ is either an arithmetic progression, or an arithmetic progression with the second-to-last point removed.
Similarly, $V_1 \cup V_2$ is either an arithmetic progression or an arithmetic progression with the second point removed.
By technical Lemma \ref{Elizabeth} below, $V$ is either an arithmetic progression or else an arithmetic progression with either the second point or the second-to-last point (but not both) removed, and these latter two cases can only occur if $3 \mid n$.
Conversely, Lemma \ref{Elizabeth} shows that sets $V$ with these forms really are optimal for $4$-term arithmetic progressions.

Now we analyze all optimal sets for $k$-term arithmetic progressions where $k \geq 5$. For $n \geq k$, let $\cV=\{V_1 < V_2 < \ldots < V_{k-1}\}$ be an $\leq$-orderly, balanced $(k-1)$-decomposition of a $n$-subset $V$ of $\R$, where we insist $\card{V_2} \geq 2$. 
Then $V' = V_1 \cup \ldots \cup V_{k-2}$ and $V'' = V_2 \cup \ldots \cup V_{k-1}$ are optimal for $(k-1)$-term arithmetic progressions by Lemma \ref{minusOneOptimal}. 
Hence, without loss of generality $V'$ is either an arithmetic progression of consecutive integers, or what one obtains by removing either the second or second-to-last point (but not both) from such an arithmetic progression.
If the rightmost two points of $V_{k-2}$ had spacing $2$, then $V''$ would begin with two points at spacing $1$, then later would have a spacing of $2$ between consecutive points, neither of which would be endpoints, and this would contradict the optimality of $V''$ for $(k-1)$-term arithmetic progressions.
So any spacing of $2$ between consecutive points in $V'$ must be between the leftmost two points. 
And by a similar argument with $V''$, we see that all consecutive points of $V''$ have spacing $1$, except possibly the two rightmost, which can have spacing $2$.
Thus all consecutive points of $V$ must have spacing $1$, except for possible spacings of $2$ between the leftmost pair and the rightmost pair.
By Lemma \ref{Elizabeth}, $V$ is either an arithmetic progression or else an arithmetic progression with either the second point or the second-to-last point (but not both) removed, and these latter two cases can only occur if $k-1 \mid n$).
Conversely, Lemma \ref{Elizabeth} shows that sets $V$ with these forms really are optimal for $k$-term arithmetic progressions.
\end{proof}
We close with the technical lemmata used to prove the two propositions above.
\begin{lemma}\label{Hannah}
Let $V$ be a $n$-subset of $\R$ that is optimal for $3$-term arithmetic progressions.
Then either the $(n-1)$-set obtained by removing the leftmost point of $V$ or the $(n-1)$-set obtained by removing the rightmost point of $V$ is optimal for $3$-term arithmetic progressions.
\end{lemma}
\begin{proof}
Let $\cV=\{V_1 < V_2\}$ be a $\leq$-orderly, balanced $2$-decomposition of $V$, with $\card{V_1} \geq \card{V_2}$.
By Lemma \ref{Francis}, we see that any pair of points in $V_1$ (resp., $V_2$) are the leftmost (resp., rightmost) two points of a $3$-term progression in $V$.

Suppose that the leftmost point $\ell$ of $V$ is such that there is no $3$-term progression $\ell < u < v$ with $u,v \in V_2$.
Then let $V^\prime=V\smallsetminus\{\ell\}$ with $\leq$-orderly, balanced $2$-decomposition $\cV^\prime=\{V_1\smallsetminus\{\ell\}, V_2\}$, and we see that any pair of points in $V_1\smallsetminus\{\ell\}$ (resp., $V_2$) are the leftmost (resp., rightmost) two points of a $3$-term progression in $V^\prime$, so $V^\prime$ is optimal by Lemma \ref{Francis}.

Now suppose that the leftmost point $\ell$ of $V$ is such that $\ell < m < r$ is a $3$-term progression with $m, r \in V_2$.
Then $m$ and $r$ must respectively be the leftmost and rightmost points of $V_2$, for otherwise Lemma \ref{Francis} would dictate that the leftmost and rightmost points of $V_2$ would be a part of a $3$-term progression in $V$, and that progression would need to involve a point to the left of $\ell$, which is absurd.

We set $V^\prime=V\smallsetminus\{r\}$, and claim that it is optimal for $3$-term arithmetic progressions.
For any $u < v \in V_1$, we claim that the point $w$ that makes $u < v < w$ a $3$-term arithmetic progression is in $V^\prime$.
For by Lemma \ref{Francis}, $w \in V$, and $w=2 v - u$, with  $v < m$ and $u \geq \ell$, so that $w < 2 m - \ell=r$ (since $\ell < m < r$ is an arithmetic progression).

If $n$ is even, then set $\cV^\prime=\{V_1,V_2\smallsetminus\{r\}\}$, which is a $\leq$-orderly, balanced $2$-decomposition of $V^\prime$, and we see that any pair of points in $V_1$ (resp., $V_2\smallsetminus\{r\}$) are the leftmost (resp., rightmost) two points of a $3$-term progression in $V^\prime$, so $V^\prime$ is optimal by Lemma \ref{Francis}.

If $n$ is odd, let $y$ be the rightmost point of $V_1$, and set $\cW=\{W_1 < W_2\}$ with $W_1=V_1\smallsetminus\{y\}$ and $W_2=V_2\cup\{y\}$, which is a $\leq$-orderly, balanced $2$-decomposition of $V$ with $\card{W_2}>\card{W_1}$.
Then set $\cW^\prime=\{W_1,W_2\smallsetminus\{r\}\}$, which is a $\leq$-orderly, balanced $2$-decomposition of $V^\prime=V\smallsetminus\{r\}$.
By Lemma \ref{Francis}, any pair $v < w$ of points in $W_2$ form the rightmost two points of a $3$-term arithmetic progression in $V$, so any pair $v < w$ of points in $W_2\smallsetminus\{r\}$ for the rightmost two points of a $3$-term arithmetic progression in $V^\prime$.
By the paragraph before the previous one, we know that every pair of points in $W_1=V_1\smallsetminus\{y\}$ form the leftmost two points of a $3$-term arithmetic progression in $V^\prime$.
Thus Lemma \ref{Francis} shows that $V^\prime$ is optimal for $3$-term arithmetic progressions.
\end{proof}
\begin{lemma}\label{minusOneOptimal}
If $V$ is a finite subset of $\R$ that is optimal for $k$-term arithmetic progressions, and $\cV=\{V_1 < \cdots < V_{k-1}\}$ is a $\leq$-orderly, balanced $(k-1)$-decomposition of $V$, then both $\cup_{j=1}^{k-2} V_j$ and $\cup_{j=2}^{k-1} V_j$ are optimal for $(k-1)$-term arithmetic progressions.
\end{lemma}
\begin{proof}
Let $V^\prime=\cup_{j=1}^{k-2}$ with $\leq$-orderly, balanced $(k-2)$-decomposition $\cV^\prime=\{V_1 < \cdots < V_{k-2}\}$.
By Lemma \ref{Francis}, for a pair of points $p$ and $q$ in the subset $V_j$ of $V'$, there is a $k$-term arithmetic progression $P$ in $\cV$ that contains these points and is only of echelon $i$ in $\cV$ for $i=j$. Note that $P' = P \smallsetminus (P \cap V_{k-1})$ is an arithmetic progression of some length that is of echelon $j$ in $\cV^\prime$. 
If $P'$ has fewer than $k-1$ points, then it must be true that more than one point of $P$ is in $V_{k-1}$, which implies that $P$ is both of echelon $k-1$ and echelon $j$ in $\cV$, contradicting Lemma \ref{Francis}.
Hence $P'$ is an arithmetic progression of at least $k-1$ terms contained in $V'$. 
By Lemma \ref{Francis}, $V'$ is optimal for $k-1$-term arithmetic progressions.
By the same reasoning, $\cup_{j=2}^{k-1} V_j$ is also optimal for $(k-1)$-term arithmetic progressions.
\end{proof}
\begin{lemma}\label{Elizabeth}
Let $n \geq k \geq 4$.  Let $V$ be an $n$-subset of $\R$ that is either an arithmetic progression, or else an arithmetic progression with the second point $p$ or the second-to-last point $q$ (or both) removed.  Then $V$ is optimal for $k$-term arithmetic progressions if and only if (i) $V$ is an arithmetic progression or (ii) $k-1$ divides $n$ and $V$ is an arithmetic progression with only one of $p$ or $q$ removed.
\end{lemma}
\begin{proof}
We can assume, without loss of generality, that $V \subseteq \Z$, that the first element of $V$ is $0$, and the smallest spacing between any two consecutive elements of $V$ is $1$.  Write $n=(k-1)q + r$ with $0 \leq r < k-1$.

We already know from Theorem \ref{George} and Lemma \ref{Eustace} that if $V$ is an arithmetic progression, then it is optimal for $k$-term arithmetic progressions. 

Now we assume that $V$ is obtained from an arithmetic progression by removing only the second point, so that $V=\{0,2,3,\ldots,n\}$.
Assume $V$ is optimal for $k$-term arithmetic progressions.

If $0 < r < k-1$, let $\cV=\{V_1 < \cdots < V_{k-1}\}$ be a $\leq$-orderly, balanced $(k-1)$-decomposition of $V$, such that $V_1=\{0,2,3,\ldots,q+1\}$.
If we consider the $k$-term arithmetic progression beginning with $0$ and $q+1$, its rightmost point is $(k-1)(q+1)=(k-1)q+(k-1) > (k-1) q + r=n$, which must be contained in $V$ by Lemma \ref{Francis}, which is absurd.

On the other hand, if $r=0$, then it is not difficult to use Lemma \ref{Francis} to show that $V=\{0,2,3,\ldots,(k-1)q\}$ is in fact optimal for $k$-term arithmetic progressions.
By symmetry, we have also covered the cases where $V$ is obtained from an arithmetic progression by only removing the second-to-last point.

Now we assume that $v$ is obtained from an arithmetic progression by removing both the second point and the second-to-last point, so that $V=\{0,2,3,\ldots,n-2,n-1,n+1\}$.
Assume $V$ is optimal for $k$-term arithmetic progressions.
Recall that we are writing $n=(k-1)q + r$ with $0 \leq r < k-1$.

If $r=0$, let $\cV=\{V_1 < \cdots < V_{k-1}\}$ be a $\leq$-orderly, balanced $(k-1)$-decomposition of $V$, in which $V_1=\{0,2,3,\ldots,q\}$.
If we consider the $k$-term arithmetic progression beginning with $0$ and $q$, then its rightmost point is $(k-1)q=n$, so Lemma \ref{Francis} says that $n$ must lie in $V$, which is absurd.

If $0 < r < k-1$, let $\cV=\{V_1 < \cdots < V_{k-1}\}$ be a $\leq$-orderly, balanced $(k-1)$-decomposition of $V$, such that $V_1=\{0,2,3,\ldots,q\}$ and $V_2=\{q+1,q+2,\ldots,2 q+1\}$.
If we consider the $k$-term arithmetic progression whose second and third points are $q+1$ and $2 q+1$, then the first point must be $1$, so Lemma \ref{Francis} says that $1$ must lie in $V$, which is absurd.
\end{proof}

\section{Commensurable Patterns in the Line}\label{Rebecca}

Here the universe $U$ is the line $\R$, and our equivalence relation $\sim$ is direct similarity.
A {\it commensurable pattern} $\cP$ in $\R$ is one such that for $P \in \cP$, all the distances between pairs in $P$ are commensurable, that is, are related by rational ratios.
Equivalently, $\cP$ is commensurable if it contains some $P \subseteq \Z$.
Or yet again, $\cP$ is commensurable if its instances are subsets of arithmetic progressions.
Indeed if $P \in \cP$, there is a unique arithmetic progression $A$ of minimum cardinality such that $P \subseteq A$.
We call this $A$ the {\it enveloping arithmetic progression} for $P$, and of course the set of all enveloping arithmetic progressions of elements of $\cP$ is itself a pattern, called the {\it enveloping pattern for $\cP$}.
For a positive $\ell$, we let $\cAl$ be the pattern consisting of $\ell$-term arithmetic progressions.

\begin{theorem}\label{Jacob}
Let $\cP$ be a commensurable $k$-pattern on $\R$, and suppose that $\cAl$ is the enveloping pattern for $\cP$.
If $n \in \N$ and $r$ and $s$ are respectively the remainders when $n$ is divided by $k-1$ and $\ell-1$, then
\[
\frac{(n-s)(n+s-\ell+1)}{2 \ell-2} \leq \SPn \leq \frac{(n-r)(n+r-k+1)}{2 k-2}.
\]
\end{theorem}
\begin{proof}
The right side of the inequality follows directly from Lemma \ref{Francis}, which gives an upper bound on the maximum number of instances of a $k$-pattern $\cP$ in an $n$-subset of $\R$.

Now let $A=\{p \leq \cdots \leq q\}$ and $A'=\{p'\leq \cdots\leq q'\}$ be distinct $\ell$-term arithmetic progressions in $R$. Then they are enveloping arithmetic progressions for distinct instances of the commensurable pattern $\cP$: $P=\{p \leq \cdots \leq q\}$ and $P'=\{p'\leq \cdots \leq q'\}$. We see that in a given set, there are at least as many instances of $\cP$ as there are $\ell$-term arithmetic progressions, so that $\SAln = (n-s)(n+s-\ell+1)/(2 \ell-2) \leq \SPn$.
\end{proof}

We now explore the commensurable pattern $\cP$ containing $\{0,1,3\}$. The enveloping pattern is $\cAfour$. We know from Theorem \ref{Jacob} that $\SAfourn  \leq \SPn \leq \SAthreen$. Note from Theorem \ref{George} that $\lim_{n\to\infty} \SAfourn/n^2 = 1/6$ and $\lim_{n\to\infty} \SAthreen/n^2=1/4$.

We now construct a family of sets $V_n$ where each $V_n$ is a set with $n$ points containing $\SPVn=(3 n^2-8 n)/16$ directly similar copies of $\{0,1,3\}$.
This makes $\lim_{n \to \infty} \SPVn/n^2=3/16$, which is strictly between the two limits we computed in the previous paragraph using the lower and upper bounds on $\SPVn$ furnished by Theorem \ref{Jacob}.
We now describe our construction $V_n$, and then prove our claim that $\SPVn=(3 n^2-8 n)/16$.
\begin{construction}\label{Mary}
If $n=96 k$ for any positive integer $k$, we let $V_n$ be the union of the following four sets:
\begin{align*}
M_0 & = 6 \Z \cap [0,108 k], \\
M_1 & = (1+6 \Z) \cap [72 k+1,144 k -5], \\
M_3 & = (3+6 \Z) \cap [3,324 k -9], \text{and} \\
M_5 & = (5+6 \Z) \cap [72 k-1,144 k -7],
\end{align*}
or equivalently
\begin{align*}
M_0 & = \{6 a: 0 \leq a \leq 18 k\},\\
M_1 & = \{6 a + 1 : 12 k \leq a \leq 24 k-1\},\\
M_3 & = \{6 a + 3 : 0 \leq a \leq 54 k -2\}, \text{and} \\
M_5 & = \{6 a - 1 : 12 k \leq a \leq 24 k-1\}.
\end{align*}
\end{construction}
We show the smallest example of this construction in Figure \ref{David}, with points drawn as vertical strokes to make them easier to discern.
\FloatBarrier
\begin{center}
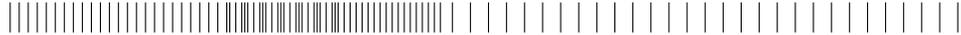
\begin{figure}[ht!]
\begin{center}
\begin{tikzpicture}[scale=0.04]
\draw (0,-5)--(0,5);
\draw (6,-5)--(6,5);
\draw (12,-5)--(12,5);
\draw (18,-5)--(18,5);
\draw (24,-5)--(24,5);
\draw (30,-5)--(30,5);
\draw (36,-5)--(36,5);
\draw (42,-5)--(42,5);
\draw (48,-5)--(48,5);
\draw (54,-5)--(54,5);
\draw (60,-5)--(60,5);
\draw (66,-5)--(66,5);
\draw (72,-5)--(72,5);
\draw (78,-5)--(78,5);
\draw (84,-5)--(84,5);
\draw (90,-5)--(90,5);
\draw (96,-5)--(96,5);
\draw (102,-5)--(102,5);
\draw (108,-5)--(108,5);
\draw (73,-5)--(73,5);
\draw (79,-5)--(79,5);
\draw (85,-5)--(85,5);
\draw (91,-5)--(91,5);
\draw (97,-5)--(97,5);
\draw (103,-5)--(103,5);
\draw (109,-5)--(109,5);
\draw (115,-5)--(115,5);
\draw (121,-5)--(121,5);
\draw (127,-5)--(127,5);
\draw (133,-5)--(133,5);
\draw (139,-5)--(139,5);
\draw (3,-5)--(3,5);
\draw (9,-5)--(9,5);
\draw (15,-5)--(15,5);
\draw (21,-5)--(21,5);
\draw (27,-5)--(27,5);
\draw (33,-5)--(33,5);
\draw (39,-5)--(39,5);
\draw (45,-5)--(45,5);
\draw (51,-5)--(51,5);
\draw (57,-5)--(57,5);
\draw (63,-5)--(63,5);
\draw (69,-5)--(69,5);
\draw (75,-5)--(75,5);
\draw (81,-5)--(81,5);
\draw (87,-5)--(87,5);
\draw (93,-5)--(93,5);
\draw (99,-5)--(99,5);
\draw (105,-5)--(105,5);
\draw (111,-5)--(111,5);
\draw (117,-5)--(117,5);
\draw (123,-5)--(123,5);
\draw (129,-5)--(129,5);
\draw (135,-5)--(135,5);
\draw (141,-5)--(141,5);
\draw (147,-5)--(147,5);
\draw (153,-5)--(153,5);
\draw (159,-5)--(159,5);
\draw (165,-5)--(165,5);
\draw (171,-5)--(171,5);
\draw (177,-5)--(177,5);
\draw (183,-5)--(183,5);
\draw (189,-5)--(189,5);
\draw (195,-5)--(195,5);
\draw (201,-5)--(201,5);
\draw (207,-5)--(207,5);
\draw (213,-5)--(213,5);
\draw (219,-5)--(219,5);
\draw (225,-5)--(225,5);
\draw (231,-5)--(231,5);
\draw (237,-5)--(237,5);
\draw (243,-5)--(243,5);
\draw (249,-5)--(249,5);
\draw (255,-5)--(255,5);
\draw (261,-5)--(261,5);
\draw (267,-5)--(267,5);
\draw (273,-5)--(273,5);
\draw (279,-5)--(279,5);
\draw (285,-5)--(285,5);
\draw (291,-5)--(291,5);
\draw (297,-5)--(297,5);
\draw (303,-5)--(303,5);
\draw (309,-5)--(309,5);
\draw (315,-5)--(315,5);
\draw (77,-5)--(77,5);
\draw (83,-5)--(83,5);
\draw (89,-5)--(89,5);
\draw (95,-5)--(95,5);
\draw (101,-5)--(101,5);
\draw (107,-5)--(107,5);
\draw (113,-5)--(113,5);
\draw (119,-5)--(119,5);
\draw (125,-5)--(125,5);
\draw (131,-5)--(131,5);
\draw (137,-5)--(137,5);
\draw (143,-5)--(143,5);
\end{tikzpicture}
\caption{Construction \ref{Mary} with $k=1$ (points drawn as vertical strokes)}\label{David}
\end{center}
\end{figure}
\end{center}
\begin{proposition}
If $\cP$ is the pattern containing $\{0,1,3\}$, and $V_n$ is the $n$-point set $V_n$ described in Construction \ref{Mary}, then $\SPVn=(3 n^2-8 n)/16$.
\end{proposition}
\begin{proof}
To compute $\SPVn$ we first note that by choosing the congruence classes modulo $6$ of the first two points of any $P \in \cP$, the class of the third point of $P$ is fixed.
For example, if we choose the first two points congruent to $5 \pmod{6}$, then the third point must also be of the same class. 
Hence the $4 \times 4$ cases for the respective congruence classes of the three points are $(0,0,0)$, $(0,1,3)$, $(0,3,3)$, $(0,5,3)$, $(1,0,4)$, $(1,1,1)$, $(1,3,1)$, $(1,5,1)$, $(3,0,0)$, $(3,1,3)$, $(3,3,3)$, $(3,5,3)$, $(5,0,2)$, $(5,1,5)$, $(5,3,5)$, and $(5,5,5)$. 
Henceforth we omit $(1,0,4)$ and $(5,0,2)$, as none of the points in $V_n$ are congruent to $4$ or to $2 \pmod{6}$.

To compute the number of instances $\{p_1 < p_2 < p_3\}$ of $\cP$ in $V_n$ with $(p_1,p_2,p_3) \equiv (a,b,c) \pmod{6}$ for a given triple $(a,b,c)$, one counts lattice points in regions of $\Z^2$ defined by a system of inequalities.
For example, if $n=96 k$ and $(a,b,c)=(0,5,3)$, we could represent each instance $\{p_1 < p_2 < p_3\}$ of $\cP$ with 
\begin{align*}
p_1 & =6 x \\
p_2 & =6 y-1 \\
p_3 & = 3 p_2-2 p_1=6(3 x - 2 y-1)+3,
\end{align*}
subject to the inequalities

\begin{align*}
x &<   y, \\
0 &\leq  x \leq 18 k, \\
12 k & \leq y \leq 24 k - 1,\text{ and} \\
0 & \leq 3 y-2 x - 1 \leq 54 k - 2.
\end{align*}

The first inequality makes sure that $p_1 < p_2$, and the remaining
three inequalities make sure that $p_1$, $p_2$, and $p_3$ are in the
ranges prescribed for their respective congruence classes, as
described in Construction \ref{Mary}.

In this way we calculate the following number of instances $\{p_1 <
p_2 < p_3\}$ of $\cP$ in $V_n$. The results for every instance with
the corresponding triple $(p_1,p_2,p_3)\pmod 6$ are described in the
following table:
\begin{center}
\begin{tabular}{|c|c|}
  \hline
  number of
  instances &  $(p_1,p_2,p_3)\pmod 6$ \\ \hline
  $54 k^2 -3 k$ & $(0,0,0)$ \\
  $171 k^2 + 6 k$ & $(0,1,3)$ \\
  $270 k^2 + 9 k$ & $(0,3,3)$ \\
  $171 k^2 + 3 k -1$ & $(0,5,3)$ \\
  $24 k^2 -6 k$ & $(1,1,1)$ \\
  $24 k^2 + 2 k$ & $(1,3,1)$ \\
  $24 k^2 -2 k$ & $(1,5,1)$ \\
  $54 k^2 + 3 k$ & $(3,0,0)$ \\
  $189 k^2 -6 k$ & $(3,1,3)$ \\
  $486 k^2 -45 k + 1$ & $(3,3,3)$ \\
  $189 k^2 -3 k$ & $(3,5,3)$ \\
  $24 k^2 + 2 k$ & $(5,1,5)$ \\
  $24 k^2 -2 k$ & $(5,3,5)$ \\
  $24 k^2 -6 k$ & $(5,5,5)$ \\
  \hline
\end{tabular}
\end{center}
If we add up all the counts, we get $1728 k^2 - 48 k$, or equivalently $(3 n^2-8 n)/16$ for the the total number of instances of $\cP$ in $V_n$. 
\end{proof}

\section{Equilateral Triangles in the Plane}\label{Eric}

In this section we explore $\SEn$, where $U$ is the Euclidean plane (identified with $\C$), $V$ is a finite subset, and $\cE$ is the pattern of vertices of an equilateral triangle.
Our order relation on $\C$ is lexicographic ordering: $y \prec z$ means that either (i) $\Re(y) < \Re(z)$ or (ii) $\Re(y)=\Re(z)$ and $\Im(y) < \Im(z)$.
This order relation respects addition, so if $y \prec y^\prime$ and $z \preceq z^\prime$, then $y+z \prec y^\prime+z^\prime$.

If $u$ and $v$ are distinct points in $\C$, they are vertices of only two equilateral triangles, and if we let $w$ and $w^\prime$ be the respective third vertices of these two triangles, then $u$, $w$, $v$, and $w^\prime$ are vertices of a parallelogram with center $(u+v)/2=(w+w^\prime)/2$, as shown in Figure \ref{Edwin} below.
\begin{center}
\begin{figure}[ht!]
\begin{center}
\begin{tikzpicture}
\node (u) at  (-0.52, 1.94) {};
\node (v) at  ( 0.52,-1.94) {};
\node (w) at  ( 3.34, 0.9) {};
\node (wp) at (-3.34,-0.9) {};
\node (c) at (0,0) {};
\draw (-0.52,1.94)--(3.34,0.9)--(0.52,-1.94)--(-3.34,-0.9)--(-0.52,1.94);
\draw (-0.52,1.94)--(0.52,-1.94);
\draw [dashed] (3.34,0.9)--(-3.34,-0.9);
\fill (u) circle [radius=2pt];
\fill (v) circle [radius=2pt];
\fill (w) circle [radius=2pt];
\fill (wp) circle [radius=2pt];
\fill (c) circle [radius=2pt];
\node (Lu) at  (-0.52, 2.24) {$u$};
\node (Lv) at  ( 0.52,-2.24) {$v$};
\node (Lw) at  ( 3.64, 0.9) {$w$};
\node (Lwp) at (-3.64,-0.9) {$w'$};
\node (Lc) at (0.85,-0.15) {{\tiny $(u+v)/2=$}};
\node (Lcc) at (1.05,-0.45) {{\tiny $(w+w')/2$}};
\end{tikzpicture}
\caption{Points $u$ and $v$ are vertices of two equilateral triangles, $\bigtriangleup u v w$ and $\bigtriangleup u v w'$, which together form a parallelogram with center at $(u+v)/2=(w+w')/2$.}\label{Edwin}
\end{center}
\end{figure}
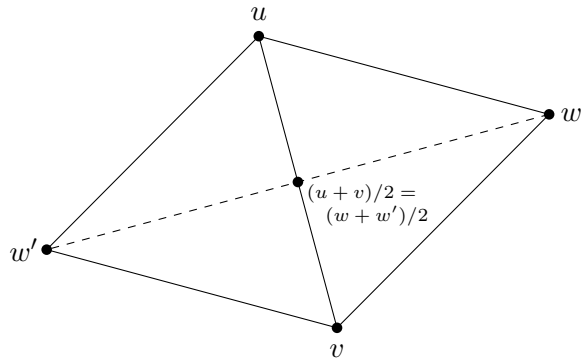
\end{center}
Then the compatibility of our order relation with addition shows that $\cE$ admits at most one reconstruction from any pair of consecutive points, for neither $u, v \prec w, w^\prime$ nor $u, v \succ w, w^\prime$ is consistent with $(u+v)/2=(w+w^\prime)/2$.

Since $\cE$ admits at most one reconstruction from any pair of consecutive points, Theorem \ref{Beatrice} immediately tells us that $\SEn \leq n(n-2)/4$ if $n$ is even and $\SEn \leq (n-1)^2/4$ if $n$ is odd, that is, $\SEn \leq \floor{(n-1)^2/4}$ in any case.
This observation recovers upper bound on $\SEn$ of \cite[Theorem 1]{Abrego-Fernandez-2000} (given here in \eqref{William}), which was previously obtained using a  much more sophisticated geometric proof.

We shall get an improved upper bound on $\SEn$ by noting that sometimes $\cE$ admits no reconstruction from a pair of consecutive points.
To see this, we define the usual principal value of the argument function $\Arg \colon \C^* \to (-\pi,\pi]$.
\begin{lemma}\label{Jacqueline}
Let $u,v$ be distinct points in $\C$.  Then these are the first two vertices (under $\preceq$) of precisely one equilateral triangle and last two vertices of precisely one equilateral triangle if $\Arg(v-u) \in (-5\pi/6,-\pi/6] \cup (\pi/6,5\pi/6]$.  Otherwise, they are neither the first two nor the last two vertices of any equilateral triangle.
\end{lemma}
One can visualize the result of Lemma \ref{Jacqueline}: there are always two equilateral triangles with vertices at $u$ and $v$, but as we vary $\Arg(v-u)$ in Figure \ref{Fiona} below, we can see which arguments make $u$ and $v$ the first and second points (under our lexicographic ordering) of one of these equilateral triangles.  Equilateral triangles for which $u$ and $v$ are the first and second points are drawn in black, while the rest are drawn in grey.
\begin{center}
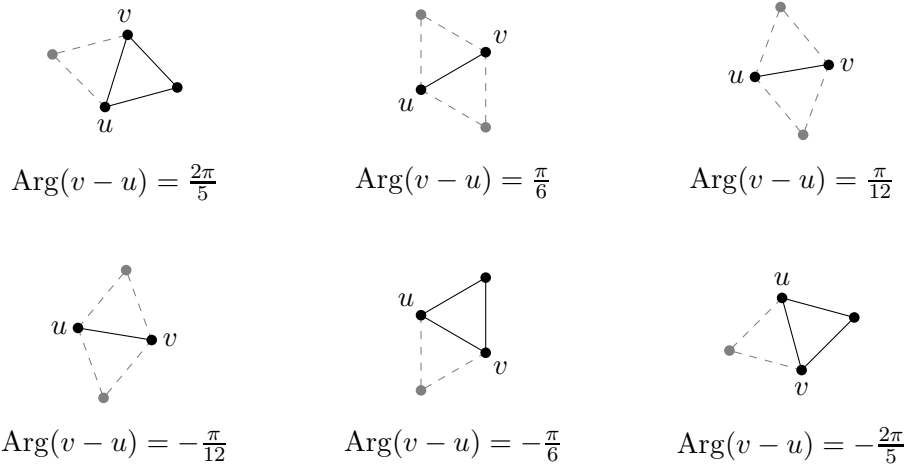
\begin{figure}[ht!]
\begin{center}
\begin{tikzpicture}
\node (u1) at  (-4.63,3.02) {};
\node (v1) at  (-4.33,3.98) {};
\node (w1) at  (-3.67,3.28) {};
\node (wp1) at (-5.33,3.72) {};
\draw [gray,dashed] (-4.63,3.02)--(-5.33,3.72)--(-4.33,3.98);
\draw (-4.63,3.02)--(-3.67,3.28)--(-4.33,3.98)--(-4.63,3.02);
\fill (u1) circle [radius=2pt];
\fill (v1) circle [radius=2pt];
\fill (w1) circle [radius=2pt];
\fill (wp1) [gray] circle [radius=2pt];
\node (Lu1) at  (-4.63,2.77) {$u$};
\node (Lv1) at  (-4.37,4.23) {$v$};
\node (arg1) at (-4.5,2.02) {$\Arg(v-u)=\frac{2\pi}{5}$};
\node (u2) at  (-0.43,3.25) {};
\node (v2) at  ( 0.43,3.75) {};
\node (w2) at  ( 0.43,2.75) {};
\node (wp2) at (-0.43,4.25) {};
\draw [gray,dashed] (-0.43,3.25)--(-0.43,4.25)--(0.43,3.75);
\draw [gray,dashed] (-0.43,3.25)--(0.43,2.75)--(0.43,3.75);
\draw (-0.43,3.25)--(0.43,3.75);
\fill (u2) circle [radius=2pt];
\fill (v2) circle [radius=2pt];
\fill (w2) [gray] circle [radius=2pt];
\fill (wp2) [gray] circle [radius=2pt];
\node (Lu2) at  (-0.63,3.05) {$u$};
\node (Lv2) at  ( 0.63,3.95) {$v$};
\node (arg2) at (0.0,2.02) {$\Arg(v-u)=\frac{\pi}{6}$};
\node (u3) at  (4.01,3.42) {};
\node (v3) at  (4.99,3.58) {};
\node (w3) at  (4.65,2.65) {};
\node (wp3) at (4.35,4.35) {};
\draw [gray,dashed] (4.01,3.42)--(4.35,4.35)--(4.99,3.58);
\draw [gray,dashed] (4.01,3.42)--(4.65,2.65)--(4.99,3.58);
\draw (4.01,3.42)--(4.99,3.588);
\fill (u3) circle [radius=2pt];
\fill (v3) circle [radius=2pt];
\fill (w3) [gray] circle [radius=2pt];
\fill (wp3) [gray] circle [radius=2pt];
\node (Lu3) at  (3.76,3.42) {$u$};
\node (Lv3) at  (5.24,3.58) {$v$};
\node (arg3) at (4.5,2.02) {$\Arg(v-u)=\frac{\pi}{12}$};
\node (u4) at  (-4.99, 0.08) {};
\node (v4) at  (-4.01,-0.08) {};
\node (w4) at  (-4.35, 0.85) {};
\node (wp4) at (-4.65,-0.85) {};
\draw [gray,dashed] (-4.99,0.08)--(-4.65,-0.85)--(-4.01,-0.08);
\draw [gray,dashed] (-4.99,0.08)--(-4.35,0.85)--(-4.01,-0.08);
\draw (-4.99,0.08)--(-4.01,-0.08);
\fill (u4) circle [radius=2pt];
\fill (v4) circle [radius=2pt];
\fill (w4) [gray] circle [radius=2pt];
\fill (wp4) [gray] circle [radius=2pt];
\node (Lu4) at  (-5.24, 0.08) {$u$};
\node (Lv4) at  (-3.76,-0.08) {$v$};
\node (arg4) at (-4.5,-1.48) {$\Arg(v-u)=-\frac{\pi}{12}$};
\node (u5) at  (-0.43, 0.25) {};
\node (v5) at  ( 0.43,-0.25) {};
\node (w5) at  ( 0.43, 0.75) {};
\node (wp5) at (-0.43,-0.75) {};
\draw [gray,dashed] (-0.43,0.25)--(-0.43,-0.75)--(0.43,-0.25);
\draw (-0.43,0.25)--(0.43,0.75)--(0.43,-0.25)--(-0.43,0.25);
\fill (u5) circle [radius=2pt];
\fill (v5) circle [radius=2pt];
\fill (w5) circle [radius=2pt];
\fill (wp5) [gray] circle [radius=2pt];
\node (Lu5) at  (-0.63, 0.45) {$u$};
\node (Lv5) at  ( 0.63,-0.45) {$v$};
\node (arg5) at (0,-1.48) {$\Arg(v-u)=-\frac{\pi}{6}$};
\node (u6) at  (4.37, 0.48) {};
\node (v6) at  (4.63,-0.48) {};
\node (w6) at  (5.33, 0.22) {};
\node (wp6) at (3.67,-0.22) {};
\draw [gray,dashed] (4.37,0.48)--(3.67,-0.22)--(4.63,-0.48);
\draw (4.37,0.48)--(5.33,0.22)--(4.63,-0.48)--(4.37,0.48);
\fill (u6) circle [radius=2pt];
\fill (v6) circle [radius=2pt];
\fill (w6) circle [radius=2pt];
\fill (wp6) [gray] circle [radius=2pt];
\node (Lu6) at  (4.37, 0.73) {$u$};
\node (Lv6) at  (4.63,-0.73) {$v$};
\node (arg6) at (4.5,-1.48) {$\Arg(v-u)=-\frac{2\pi}{5}$};
\end{tikzpicture}
\caption{Equilateral triangles with vertices at $u$ and $v$: in solid black are those that make $u$ and $v$ the least two vertices under lexicographical ordering $\preceq$.}\label{Fiona}
\end{center}
\end{figure}
\end{center}
\noindent{\it Proof of Lemma \ref{Jacqueline}:}
Let $\zeta_6=e^{\pi i/3}$.
Then $u$ and $v$ are vertices of only two equilateral triangles, with the third point being either (i) $w=\zeta_6 u+\conj{\zeta_6} v$ or (ii) $w^\prime=\conj{\zeta_6} u + \zeta_6 v$.

Note that $u, v \prec w$ requires $\Re(w) \geq \Re(u)$ (with $\Im(w) > \Re(u)$ in case of equality) and $\Re(w) \geq \Re(v)$ (with $\Im(w) > \Im(v)$ in case of equality).
Since $\zeta_6+\conj{\zeta_6}=1$, this is equivalent to $\Re(\conj{\zeta_6}(v-u)) \geq 0$ and $\Re(-\zeta_6(v-u)) \geq 0$, with the corresponding imaginary parts strictly positive in cases of equality.
These conditions are fulfilled if and only if $\Arg(v-u) \in (\pi/6,5\pi/6]$.
One can similarly show that $u, v \succ w^\prime$ under the same conditions, and that $u, v \prec w^\prime$ if and only if $\Arg(v-u) \in (-5\pi/6,-\pi/6]$, with $u, v \succ w$ under the same conditions.\hfill\qedsymbol
\vskip 5mm
Let $A=(-5\pi/6,-\pi/6] \cup (\pi/6,5\pi/6]$, the set of values of $\Arg(v-u)$ that admit a reconstruction of $\cE$ with $u$ and $v$ as consecutive points.
To get a good upper bound for $\SEV$ in an $n$-set $V$, we rotate $V$ in such a way that the first $\ceil{n/2}$  points of $V$ (with respect to ordering $\preceq$) have many pairs $(u, v)$ of points with $\Arg(v-u) \not\in A$, and the last $\floor{n/2}$ points of $V$ have many pairs $(u,v)$ of points with $\Arg(u-v) \not\in A$.
Thus when we decompose our set $V$ into a balanced $\preceq$-orderly $2$-decomposition $\cV=\{V_1,V_2\}$, many of the pairs in $V_1$ admit no reconstruction of which they are the first two points, and likewise may of the pairs in $V_2$ admit no reconstruction of which they are the last two points, so that we may deduct the counts of these ``unproductive'' pairs from the standard upper bound of $\floor{(n-1)^2/4}$ implied by Theorem \ref{Beatrice}.
This requires some care, for as we rotate, points change from being among the first $\ceil{n/2}$ points of $V$ to being among the last $\floor{n/2}$ points of $V$ and vice-versa, that is, our decomposition $\cV=\{V_1,V_2\}$ must change (points migrate between $V_1$ and $V_2$) as we rotate $V$.
We now indicate roughly how we overcome this difficulty, with technical details given later.  We eventually shall show that, up to translation and rotation, we may assume that (i) approximately half of the points in $V$ lie on each side of the $y$-axis, and (ii) approximately half of the points in $V$ lie on each side of the line $y=x/\sqrt{3}$, and (iii) approximately half of the points in $V$ lie on each side of the line $y=-x/\sqrt{3}$.  These three lines cut the plane into six compartments, and we call pairs of points lying within the same compartment {\it intracompartmental pairs}.  We then rotate this picture about the origin by angles of $0$, $2\pi/3$, and $4\pi/3$.  Each intracompartmental pair $(u,v)$ will have $\Arg(v-u) \in A$ for two of these rotations, and $\Arg(v-u)\not\in A$ for one of these rotations.  Thus we may choose a rotation such that at least one-third of the intracompartmental pairs $u,v$ have $\Arg(v-u) \not\in A$.  Sometimes points lie on the boundaries: such technical issues can be handled naturally using the formalisms that follow.

We define a {\it direction $\zeta$} to be complex number of unit modulus, so the set of directions is the complex unit circle.  A {\it directed line} with direction $\zeta$ is line with parameterization $t\mapsto \zeta t + \eta$ for some $\eta \in \C$, where the parameter $t$ ranges over $\R$.
Every directed line with direction $\zeta$ may be uniquely written as $L(t)=\zeta t + s (-i\zeta)$, where $s$ is a real scalar uniquely determined by the line $L$; we call this the {\it canonical form} of the directed line $L$.
Figure \ref{Grace} shows the relation between the directions $\zeta$ and $-i\zeta$, and the parameterization of the directed line $L(t)$.
Note that $|s|$ is the distance from $L$ to the origin, and the sign of $s$ tells us which side of $L$ the origin lies on.
We may thus identify the set of directed lines with the Cartesian product of the unit circle (for $\zeta$) and the real line (for $s$), and give the directed lines the topology of this product space.
\begin{center}
\begin{figure}[ht!]
\begin{center}
\begin{tikzpicture}[>=triangle 45]
\draw[->,gray] (0,0)--(0.707,0.707);
\draw[->,gray] (0,0)--(0.707,-0.707);
\node (ReL) at (-3.535,0) {};
\node (ReR) at (4.0,0) {$\Re$};
\node (ImB) at (0,-1) {};
\node (ImT) at (0,4.950) {$\Im$};
\draw [->] (ImB)--(ImT);
\draw [->] (ReL)--(ReR);
\node (zero) at (0,0) {};
\fill (zero) circle [radius=2pt];
\node (zerot) at (-0.3,-0.3) {$0$};
\node (zeta) at (0.707,0.707) {};
\fill (zeta) circle [radius=2pt];
\node (zetat) at (1.007,0.707) {$\zeta$};
\node (nizeta) at (0.707,-0.707) {};
\fill (nizeta) circle [radius=2pt];
\node (zetat) at (1.207,-0.707) {$-i \zeta$};
\draw [gray] (0.17675,0.17675)--(0.3535,0)--(0.17675,-0.17676);
\node (L0) at (-1.414,1.414) {};
\fill (L0) circle [radius=2pt];
\node (L0t) at (-2.914,1.414) {$L(0)=s(-i\zeta)$};
\draw [dashed] (0,0)--(-1.414,1.414);
\draw (-1.23725,1.59075)--(-1.0605,1.414)--(-1.23725,1.23725);
\draw [decorate,decoration={brace,amplitude=5pt,raise=4pt},yshift=0pt] (0,0)--(-1.414,1.414) node [black,midway,xshift=-14.0,yshift=-10.0] {$|s|$};
\node (LR) at (2.114,4.950) {$L$};
\node (LL) at (-3.535,-0.707) {};
\node (Lt) at (0.707,3.535) {};
\fill (Lt) circle [radius=2pt];
\node (Ltt) at (2.707,3.535) {$L(t)=\zeta t + s (-i \zeta)$};
\draw[->]  (LL)--(LR);
\draw [decorate,decoration={brace,amplitude=5pt,raise=4pt},yshift=0pt] (-1.414,1.414)--(0.707,3.535) node [black,midway,xshift=-14.0,yshift=10.0] {$|t|$};
\end{tikzpicture}
\caption{The directed line $L(t)=\zeta t + s(-i\zeta)$ (with constants $s \in \R$ and $\zeta \in \C$ with $|\zeta|=1$) has direction $\zeta$ and is distance $|s|$ from the origin.  In this particular instance, $s < 0$, which means that the origin is to the right of a viewer standing on the line and facing in its direction.}\label{Grace}
\end{center}
\end{figure}
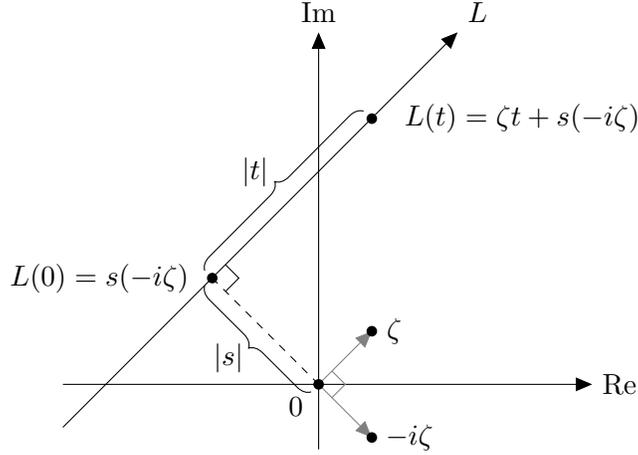
\end{center}

For each direction $\zeta$, we define a total ordering relation $\preceq_\zeta$ on $\C$ with $y \preceq_\zeta z$ if and only if $i\conj{\zeta} y \preceq i\conj{\zeta} z$, where $\preceq$ is our usual lexicographic ordering on $\C$.
Note that $\preceq_i$ is the same as $\preceq$, and for any direction $\zeta$, we have $y \preceq_{-\zeta} z$ if and only if $y \succeq_\zeta z$.
If $L(t)$ is a directed line with direction $\zeta$, then we say $z \in \C$ lies {\it to the left} (resp., {\it to the right}) of $L(t)$ if $z \prec_\zeta L(t)$ for all $t \in \R$ (resp., $z \succ_\zeta L(t)$ for all $t \in \R$).
If $y, z \in \C$ are on the line $L(t)$, say $y=L(t_y)$ and $z=L(t_z)$, then we say that $y$ is {\it below} (resp., {\it above}) $z$ on $L$ if $t_y < t_z$ (resp., $t_y > t_z$).
Note that these concepts of left and right, below and above, all correspond to the usual notions when $L(t)=i t$ is the positively directed $y$-axis, or indeed if $L$ is any translate thereof.

If $\zeta$ is a direction and $V=\{v_1 \prec_\zeta \cdots \prec_\zeta v_n\}$ is a finite set, then the {\it $\zeta$-median of $V$} is $v_{(n+1)/2}$ if $n$ is odd, and is $(v_{n/2}+v_{n/2+1})/2$ if $n$ is even.
If $\zeta$ is a direction and $V$ is a finite set of points, then the {\it $\zeta$-halving line of $V$} is the directed line $L_{V,\zeta}$ with direction $\zeta$ that passes through the $\zeta$-median of $V$.
The line is called a {\it halving line} because $\floor{n/2}$ of the points of $V$ are (with respect to $\preceq_\zeta$) less than the median, and so lie to the left of the line or below the $\zeta$-median of $V$ on the directed line, while $\floor{n/2}$ of the points of $V$ are greater than the median, and so lie to the right of the line or above the median on the directed line.
Our definition of halving line makes a unique halving line for $V$ in each direction, and is identical to the definition used by Erickson, Hurtado, and Morin \cite{Erickson-Hurtado-Morin}.
If $c$ is the $\zeta$-median of $V$, we call any point $z$ with $z \preceq_\zeta c$ {\it formally to the left} of $L_{V,\zeta}$ and any point $z$ with $z \succ_\zeta c$ {\it formally to the right} of $L_{V,\zeta}$.
Thus $\ceil{n/2}$ of the points of $V$ lie formally to the left of $L_{V,\zeta}$ and $\floor{n/2}$ of the points of $V$ lie formally to the right of $L_{V,\zeta}$.

With the topology for the space of directed lines introduced above, we claim that $L_{V,\zeta}$ is continuous in $\zeta$.
Indeed, $L_{V,\zeta}$ evolves by rotating about a particular median for a segment of values of $\zeta$, changing from median $v$ to $w$ precisely for the value of $\zeta$ that makes $v$ and $w$ both lie on the line $L_{V,\zeta}$.
Note that the $\zeta$-median and the $(-\zeta)$-median of $V$ are the same, so that $L_{V,\zeta}$ and $L_{V,-\zeta}$ are the same line, but with opposite directions.
These considerations lead to the following useful observation, first proved in this precise form in \cite[Lemma 3]{Erickson-Hurtado-Morin} (see also \cite[Lemma 2]{Fekete-Meijer} and \cite[Lemma 2]{Dumitrescu-Pach-Toth}).
\begin{lemma}\label{halvingLines}
Let $V$ be a finite subset of $\C$.
Then there is some direction $\zeta$ such that the halving lines of $V$ in directions $\zeta$, $e^{2\pi i/3} \zeta$, and $e^{4\pi i/3}\zeta$ are concurrent.
\end{lemma}
\begin{proof}
Let $\omega=e^{2\pi i/3}$.
If the point $L_{V,\omega}\cap L_{V,\omega^2}$ is on or to the left (resp., right) of the directed line $L_{V,1}$, then the same point $L_{V,-\omega} \cap L_{V,-\omega^2}$ must be on or to the right (resp., left) of the same but oppositely directed line $L_{V,-1}$.
As $\zeta$ traverses the unit circle, the lines $L_{V,\zeta}$, $L_{V,\omega\zeta}$, and $L_{V,\omega^2\zeta}$ evolve continuously, and so the intersection point $L_{V,\omega\zeta} \cap L_{V,\omega^2\zeta}$ also evolves continuously.
Thus there must be some $\zeta$ such that $L_{V,\omega\zeta} \cap L_{V,\omega^2\zeta}$ crosses $L_{V,\zeta}$.
\end{proof}
Now we are ready to recall and prove our improved bound (Theorem \ref{Katherine}) on $\SEn$.
\begin{theorem}
If $\cE$ is the pattern of the vertices of equilateral triangles in $\R^2$, then $\SEn \leq \floor{(4 n-1)(n-1)/18}$.
\end{theorem}
\begin{proof}
Let $V$ be an $n$-subset of $\R^2$ (identified with $\C$).
Obtain three concurrent halving lines of $V$, say $L$, $M$, and $N$, as described in Lemma \ref{halvingLines}, where the directions of $M$ and $N$ are obtained from that of $L$ by rotating by $2\pi/3$ and $4\pi/3$, respectively.
This means that a point can only be formally left (or formally right) of all three lines if it lies on all three.
So we may decompose $V$ into six or seven disjoint classes of points (called {\it compartments}), where a point is classified according to its position (formally left or right) relative to the three lines.
\begin{center}
\begin{tabular}{cccc}
      & \multicolumn{3}{c}{Position Relative to Line} \\
Compartment & $L$ & $M$ & $N$ \\
\hline
\hline
$V_1$ & formally left  & formally right & formally left \\
$V_2$ & formally left  & formally right & formally right \\
$V_3$ & formally left  & formally left  & formally right \\
\hline
$V_4$ & formally right & formally left  & formally right \\
$V_5$ & formally right & formally left  & formally left \\
$V_6$ & formally right & formally right & formally left \\
\hline
\multirow{2}{*}{$V_7$} & on & on & on \\
 & \multicolumn{3}{c}{(if this point is not in $V_1 \cup \cdots \cup V_6$)}
\end{tabular}
\end{center}
Note that $V_7$ is nonempty if and only if a point of $V$ happens to lie on the intersection of the three lines and if said point happens to be formally left of all three lines or formally right of all three lines.  We depict our compartments in Figure \ref{Harold}.
\begin{center}
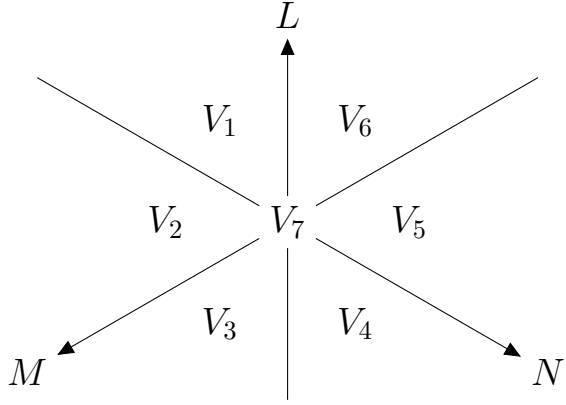
\begin{figure}[ht!]
\begin{center}
\begin{tikzpicture}[>=triangle 45]
\node (HL) at (0,2.75) {\Large{$L$}};
\node (TL) at (0,-2.5) {};
\node (HM) at (-3.464,-2) {\Large{$M$}};
\node (TM) at ( 3.464, 2) {};
\node (HN) at ( 3.464,-2) {\Large{$N$}};
\node (TN) at (-3.464, 2) {};
\node (V1) at (-0.9,1.35) {\Large{$V_1$}};
\node (V2) at (-1.62,0.0) {\Large{$V_2$}};
\node (V3) at (-0.9,-1.35) {\Large{$V_3$}};
\node (V4) at (0.9,-1.35) {\Large{$V_4$}};
\node (V5) at (1.62,0.0) {\Large{$V_5$}};
\node (V6) at (0.9,1.35) {\Large{$V_6$}};
\node (V7) at (0,0) {\Large{$V_7$}};
\draw (TL)--(V7);
\draw (TM)--(V7);
\draw (TN)--(V7);
\draw[->]  (V7)--(HL);
\draw[->]  (V7)--(HM);
\draw[->]  (V7)--(HN);
\end{tikzpicture}
\caption{The seven possible compartments, where $V_7$ can only contain the intersection point of the three lines}\label{Harold}
\end{center}
\end{figure}
\end{center}

Let $A=(-5\pi/6,-\pi/6] \cup (\pi/6,5\pi/6]$, the set of values of $\Arg(v-u)$ that (by Lemma \ref{Jacqueline}) admit a reconstruction of $\cE$ with $u$ and $v$ as consecutive points.
We call pairs of points in $\bigcup_{j=1}^6 V_j\times V_j$ {\it intracompartmental pairs}.
We can rotate $V$ and the directed lines $L$, $M$, and $N$ so that one of them is $t\mapsto i t$, the positively directed $y$-axis.
Each intracompartmental pair $(u,v)$ will have $\Arg(v-u) \in A$ for two of these rotations, and $\Arg(v-u)\not\in A$ for one of these rotations. 
Thus we may choose a rotation such that at least one-third of the intracompartmental pairs $u,v$ have $\Arg(v-u) \not\in A$.
Without loss of generality, let us say that our chosen rotation makes the line $L$ the positively directed $y$-axis.

Every instance $\{v_1 \prec v_2 \prec v_3\}$ of $\cE$ in $V$ either has $v_1$ and $v_2$ formally to the left of $L$  or has $v_2$ and $v_3$ formally to the right of $L$.
Furthermore, Lemma \ref{Jacqueline} shows that in the former case $\Arg(v_2-v_1) \in A$ and $v_3$ is uniquely determined by $v_1$ and $v_2$, while in the latter case $\Arg(v_3-v_2) \in A$ and $v_1$ is uniquely determined by $v_2$ and $v_3$.
Since $L$ is a halving line, there are $\ceil{n/2}$ points formally to the left of $L$ and $\floor{n/2}$ points formally to the right of $L$.
So we count the number of pairs $(u,v)$ in $V$ with $u$ and $v$ formally to the same side of $L$, and deduct the number of intracompartmental pairs that we deliberately rotated to prevent them from admitting a reconstruction of $\cE$ to obtain
\begin{equation}\label{Terence}
\SEV \leq \binom{\ceil{n/2}}{2}+\binom{\floor{n/2}}{2}-\frac{1}{3} \sum_{j=1}^{6} \binom{\card{V_j}}{2}.
\end{equation}
Now $\sum_{j=1}^6 \card{V_j}$ is either $n$ or $n-1$ (depending on whether $V_7$ is empty or not), and by convexity, the sum of binomial coefficients in \eqref{Terence} is minimized when the various values $\card{V_j}$ are as close to equal as possible.
So if $n-1=6 q+2 r+s$ with $q \in \Z$, $r \in \{0,1,2\}$ and $s \in \{0,1\}$, then
\begin{align}
\sum_{j=1}^{6} \binom{\card{V_j}}{2}
& \geq (2 r+s) \binom{q+1}{2} + (6-2 r -s)\binom{q}{2} \label{Helen} \\
& = q (3 q+ 2 r + s-3).\nonumber
\end{align}
In this case, we see that the first two terms of \eqref{Terence} become
\begin{align}
\binom{\ceil{n/2}}{2}+\binom{\floor{n/2}}{2}
& =  (s+1) \binom{3 q+r+1}{2} + (1-s) \binom{3 q + r}{2} \label{Peter} \\
& = (3 q+r)(3 q+r +s), \nonumber 
\end{align}
and substituting \eqref{Helen} and \eqref{Peter} into \eqref{Terence}, we obtain
\[
\SEV \leq \frac{8 q(q+2 r+s) + 3 q + 3 r(r+s)}{3},
\]
which works out to
\[
\SEV \leq \begin{cases}
\frac{2}{9} n^2 - \frac{5}{18} n - \frac{1}{3} & \text{if $n\equiv 0$ or $2 \pmod{6}$,} \\
\frac{2}{9} n^2 - \frac{5}{18} n + \frac{1}{18} & \text{if $n\equiv 1 \pmod{6}$,} \\
\frac{2}{9} n^2 - \frac{5}{18} n - \frac{1}{6} & \text{if $n\equiv 3$ or $5 \pmod{6}$,} \\
\frac{2}{9} n^2 - \frac{5}{18} n - \frac{4}{9} & \text{if $n\equiv 4 \pmod{6}$,}
\end{cases}
\]
so that $\SEV \leq \floor{(4 n-1)(n-1)/18}$ for arbitrary $n$.  Since $V$ was (up to our rotation) an arbitrary $n$-subset of $\C$, this upper bound holds for $\SEn$.
\end{proof}

\end{document}